\documentclass[12pt]{article}
\usepackage{graphicx}
\usepackage{amsmath}
\usepackage{amsfonts}
\usepackage{amsthm}
\usepackage{verbatim}
\usepackage[pdftex]{hyperref}
\usepackage{color}

\usepackage[T1]{fontenc}
\usepackage[utf8]{inputenc}

\newtheorem{theorem}{Theorem}
\newtheorem*{theorem*}{Theorem}
\newtheorem{lemma}[theorem]{Lemma}
\newtheorem{corollary}[theorem]{Corollary}

\theoremstyle{definition}

\DeclareMathOperator{\mex}{mex}

\begin{document}

\title{On Variations of Nim and Chomp}

\author{June Ahn \and Benjamin Chen \and Richard Chen \and Ezra Erives \and
Jeremy Fleming \and Michael Gerovitch \and Tejas Gopalakrishna \and Tanya Khovanova \and
Neil Malur \and 
Nastia Polina \and
Poonam Sahoo}
\date{}

\maketitle

\begin{abstract}
We study two variations of Nim and Chomp which we call \textit{Monotonic Nim} and \textit{Diet Chomp}. In Monotonic Nim the moves are the same as in Nim, but the positions are non-decreasing numbers as in Chomp. Diet-Chomp is a variation of Chomp, where the total number of squares removed is limited.
\end{abstract}

\section{Introduction}

We study finite impartial games with two players where the same moves are available to both players. Players alternate moves. In a \textit{normal} play, the person who does not have a move loses. In a \textit{mis\`{e}re} play, the person who makes the last move loses.

A \emph{P-position} is a position from which the \emph{previous} player wins, assuming perfect play. We can observe that all terminal positions are P-positions. An \emph{N-position} is a position from which the \emph{next} player wins given perfect play. When we play we want to end our move with a P-position and want to see an N-position before our move.

Every impartial game is equivalent to a Nim heap of a certain size. Thus, every game can be assigned a non-negative integer, called \textit{a nimber}, \textit{nim-value}, or \textit{a Grundy number}.

The game of Nim is played on several heaps of tokens. A move consists of taking some tokens from one of the heaps.

The game of Chomp is played on a rectangular $m$ by $n$ chocolate bar with grid lines dividing the bar into $mn$ squares. A move consists of chomping a square out of the chocolate bar along with all the squares to the right and above. The player eats the chomped squares. Players alternate moves.
The lower left square is poisoned and the player forced to eat it dies and loses.

The game of Chomp is a mis\`{e}re game. The normal game is not interesting as the first player can just eat the whole bar and win.

The game of Chomp is not completely solved \cite{Z}, but the first player wins (in a non-trivial game when $mn > 1$). This can be proven by a \textit{strategy-stealing argument}. Suppose that the second player has a winning strategy. Suppose that in the first move the first player takes only the top right square. By our assumption, the second player has a winning response to this. But such a response is a legal first move and the first player could have played it.

We study two variations of Nim and Chomp which we call \textit{Monotonic Nim} and \textit{Diet Chomp}.

\textbf{Monotonic Nim.} In this variation the players are restricted to eat only shapes that are 1 by $k$ horizontal rectangles. Equivalently, this is a variation of Nim where the number of tokens in piles must be non-decreasing through the game. This can be viewed as a merger between Nim and Chomp where the moves are the same as in Nim and the positions are restricted to Young diagrams as in Chomp.

\textbf{Diet Chomp.} In this variation of Chomp players are not allowed to eat too much chocolate in one move. That is, the number of squares that can be removed in one move is restricted by a parameter $k$. The players are allowed to make a move the same way as in the game of Chomp with a condition that they can only chomp away not more than $k$ small chocolate squares at a time. When $k$ is given, we call this variation \textit{$k$-Diet Chomp}.  Unlike regular Chomp, the normal play becomes interesting here.

We also discuss the combination of these two games, which can be called \textit{Diet Monotonic Nim}. We also call it \textit{Slow Monotonic Nim}, because Nim itself does not imply chocolate or eating. In $k$-Slow Monotonic Nim the starting position is a sequence of non-decreasing positive integers $(a_1,a_2,\ldots,a_n)$. A player can subtract up to $k$ from one of the numbers, given that the resulting sequence is non-decreasing. It is worth noting that the mis\`{e}re game for monotonic variations is equivalent to considering the last token in the last pile being poisonous.

To start, we recapitulate known facts about subtraction games.

\section{Nim, Slow Nim and Extended Nim}

In the game of \textit{Nim} there are several piles of tokens. The players are allowed to take any number of tokens from a single pile. This game started combinatorial game theory, see \cite{Bouton,BCG,AND}. The solution to Nim is well known.

Suppose $A=(a_1,a_2,\ldots,a_n)$, is a position in this game. Let us denote the XOR operation as $\oplus$. Then the following theorem is true.

\begin{theorem}
For normal play Nim, the Grundy value of a position $(a_1,a_2,\ldots,a_n)$ is 
\[a_1 \oplus a_2 \oplus \cdots \oplus a_n.\]
\end{theorem}

The P-positions correspond to Grundy value zero.

\begin{corollary}
A P-position in normal play Nim satisfies:
\[a_1 \oplus a_2 \oplus \cdots \oplus a_n = 0.\]
\end{corollary}

Similarly, the P-positions for mis\`{e}re play are known \cite{BCG}:

\begin{theorem}
For the mis\`{e}re play if $\max{a_i} > 1$, a P-position satisfies:
\[a_1 \oplus a_2 \oplus \cdots \oplus a_n =0,\]
otherwise: 
\[a_1 \oplus a_2 \oplus \cdots \oplus a_n = 1.\]
\end{theorem}

A \textit{subtraction game}, denoted Subtraction($S$), is played with heaps of tokens. A move is defined by choosing a heap and removing any number of tokens, such that this number is in set $S$.

The subtraction games are well-studied \cite{BCG,AND}, and we restrict ourselves to the case when $S$ is equal to $[k]$, where the latter denotes the range of integers from 1 to $k$ inclusive. We call this set of games \textit{Slow Nim}. For a particular $k$ we call the game \textit{k-Slow Nim}.

The Grundy values and P-positions for this game are known.
\begin{theorem}
For $k$-Slow Nim normal play, the Grundy value for a position $(a_1,a_2,\ldots,a_n)$ is 
\[(a_1 \pmod{k+1}) \oplus (a_2 \pmod{k+1}) \oplus \cdots \oplus (a_n \pmod{k+1}).\]
\end{theorem}

Therefore the P-positions are such that 
\[(a_1 \pmod{k+1}) \oplus (a_2 \pmod{k+1}) \oplus \cdots \oplus (a_n \pmod{k+1}) = 0.\]

\begin{theorem}
The P-positions in mis\`{e}re k-Slow Nim considered modulo $k+1$ are:
\begin{itemize}
\item If there is a pile that is more than 1, then XOR is zero.
\item If every pile is zero or one, then there is an odd number of ones. (XOR is 1)
\end{itemize}
\end{theorem}

In the next variation we want to allow the players to put tokens back into a pile. It seems that such a game is not finite, as an infinite loop might be created. To avoid that, we put a limit on the number of tokens that can be put back. We call this game \textit{Extended Nim}. Similarly, \textit{Extended $k$-Slow Nim} is like $k$-Slow Nim where, in addition, the players are allowed to put up to $k$ tokens back into any one of the piles, given that the total number of tokens that are put back is limited by $k$.

\begin{theorem}
The extended games have the same P-positions as the non-extended equivalents and the same Grundy values.
\end{theorem}

\begin{proof}
Consider a position $A$. Let $S$ be the set of all positions to which we can move from $A$ in a regular game and $S'$ be the positions to which we can move from $A$ in the extended version. As we add moves: $S \in S'$. Consider the sets of Grundy values $G$ for $S$ and $G'$ for $S'$ with respect to the non-extended game. On one hand, $G \in G'$. On the other hand, $G(A) \notin G'$. The latter is due to the fact that a new position $A'$ to which we can move from $A$ in the extended game have the Grundy value different from $A$ due to the fact that there is a move from $A'$ to $A$ in the non-extended game. It follows that $\mex(G) = \mex(G')$, and by definition the Grundy value of $A$ in the extended version is $G(A)$.

It follows that P-positions are the same in both games.
\end{proof}

\section{Monotonic Games}

To move from Nim to Chomp, we consider games where a position $A = (a_1,a_2,\ldots,a_n)$ is allowed only if the sequence is non-decreasing $a_i \leq a_{i+1}$, for $1 \leq i < n$. 

\textit{Monotonic Nim} is a monotonic game where you can take any number of tokens from one pile, given that the resulting sequence is non-decreasing.

If, in addition, we put a limit of $k$ on the total number of tokens that can be taken we get a game that we call \textit{Monotonic k-Slow Nim}. As in any monotonic game the only positions that are allowed are sequences of non-decreasing positive integers $(a_1,a_2,\ldots,a_n)$. A player can subtract any number of tokens between 1 and $k$ inclusive from one pile, given that the resulting sequence is non-decreasing.

Suppose we have a position $A = (a_1,a_2,\ldots,a_{2k})$ with an even number of piles. We map it a position $B = (b_1,b_2,\ldots,b_k)$, where $b_i = a_{2i} - a_{2i-1}$. For a position with an odd number of piles we first extend it to a position with an even number of piles, by adding a zero pile in front. We call the position $B$ the \textit{difference position}.

\begin{theorem}
A position $A$ is a P-position in a Monotonic game if and only if the corresponding difference position is a P-position in the corresponding extended game.
\end{theorem}

\begin{proof}
In the monotonic game, we can take any number of tokens between 1 and $a_{2i} - a_{2i-1}$ inclusive from pile $2i$. This is equivalent to taking any number of tokens between 1 and $b_i$ inclusive from pile $i$ in the corresponding difference game. In addition, in the monotonic game, we can take any number of tokens between 1 and $a_{2i-1} - a_{2i-2}$
from pile $2i-1$. This is equivalent to adding some tokens to the $i$-th pile in the differences position. Notice that the number of tokens we can add has constraints. In any case, the total number of tokens we can add is limited by $\sum\limits_{i=1}^k a_i$. We can say that the monotonic game is equivalent to playing the extended Nim with additional constraints on the difference position.

In any case, we added some extra moves to the corresponding game that do not allow to move from a P-position to a P-position. That means the set of P-positions on the difference game exactly corresponds to the P-positions in the Monotonic game.
\end{proof}

Notice that the theorem works for both normal and mis\`{e}re plays.

\section*{2-Diet Chomp Normal Play}

Now we move to Chomp for health-conscious players. Namely, we study a variation of Chomp where a player makes a Chomp move that is limited to one or two chocolate squares. The positions in our game are $A = (a_1,a_2,\ldots,a_n)$, so that the sequence is non-decreasing: $a_i \leq a_{i+1}$, where $1 \leq i  < n$. We can assume that $a_0 = 0$.

In one move we are allowed to:

\begin{itemize}
\item subtract 1 from $a_i$ if $a_i > a_{i-1}$.
\item subtract 2 from $a_i$ if $a_i > a_{i-1}+1$.
\item subtract 1 from $a_i$ and $a_{i+1}$ if $a_{i+1} = a_i > a_{i-1}$.
\end{itemize} 

For 2-Diet Chomp, the P-positions depend on the total number of tokens.

\begin{lemma}
The P-positions are such that the total number of tokens $\big(\sum\limits_{i=1}^{n}a_i\big)$ is divisible by 3.
\end{lemma}

\begin{proof}
The terminal position, $(0)$, is a P-position. P-positions differ by multiples of $3$, therefore there is no move from a P-position to a P-position. What is left to show is that all N-positions have a move to a P-position.

Suppose $\big(\sum\limits_{i=1}^{n}a_i\big) \equiv 1 \mod 3$. You can always remove one square, so it moves to a P-position. If $\big(\sum\limits_{i=1}^{n}a_i\big) \equiv 2 \mod 3$, removing two squares moves it to a P-position, except there could be a position such that there is no valid move that removes two squares. 

The only positions for which it is not allowed to remove two squares are ``perfect stairs'' positions: $(1,2,\ldots,n)$. However, the total number of tokens in such a position is a triangular number; and it is widely known that triangular numbers do not have remainder $2$ modulo 3. That means we can always move from an N-position to a P-position.
\end{proof}

Interestingly, in this case the game is equivalent to playing 2-Slow Nim on one pile.

\section*{2-Diet Chomp Mis\`{e}re play}

This game is more difficult than the normal play. 

We can explicitly describe P-positions for narrow rectangles. 

\begin{lemma}
For rectangles 1 by $n$, the P-positions are $3k+1$. For rectangles 2 by $n$ the P-positions are $(a,a+3k+1)$. Here $k \geq 0$.
\end{lemma}

\begin{proof}
For 1 by $n$ rectangles, the game is equivalent to 2-Slow Nim on one pile, mis\`{e}re play. For 2 by $n$ rectangles, $(a_1,a_2)$ is a P-position if and only if $a_2 - a_1 \equiv 1 \pmod 3$. Notice that we cannot have a move that changes both values from a P-position. By subtracting 1 or 2 from each coordinate we change the difference modulo 3. That means every move from a P-position goes to an N-position.

On the other hand, from an N-position $(a,a+3k+2)$, we can move to $(a,a+3k+1)$, which is a P-position. From an N-position $(a,a+3k)$, we can move to $(a-1,a+3k)$, which is a P-position. Additionally, if $a=0$, $(0,3k+1)$ is a P-position.
\end{proof}

For 3 by $n$ rectangles, the situation is more complicated. We wrote a program and observed that P-positions are periodic with period 12. That is, position $(a_1,a_2,a_3)$ is the same type as $(a_1+12,a_2+12,a_3+12)$. Figure~\ref{fig:DMCP-pos} shows P-positions for values $a_1$ ranging from 0 to 11 inclusive. The left bottom corner corresponds to $(a_1,a_1,a_1)$ and P-positions are black.

\begin{table}[htb]
\centering
\begin{tabular}{c c c c} 
\includegraphics[scale=0.2]{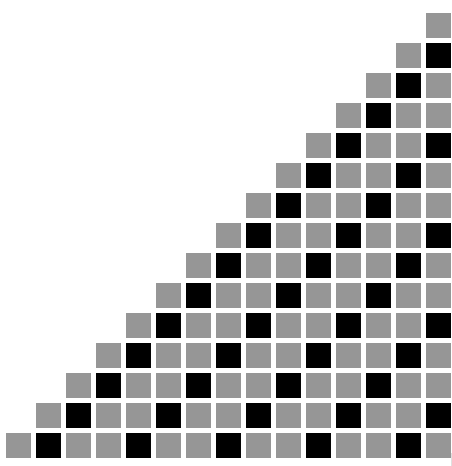} & \includegraphics[scale=0.2]{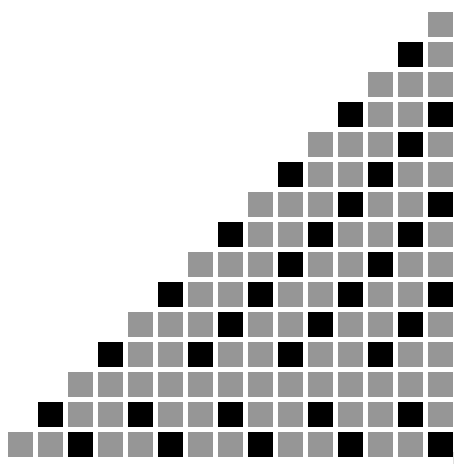} & \includegraphics[scale=0.2]{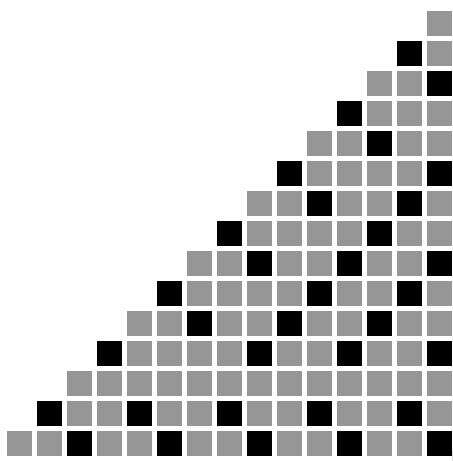} & \includegraphics[scale=0.2]{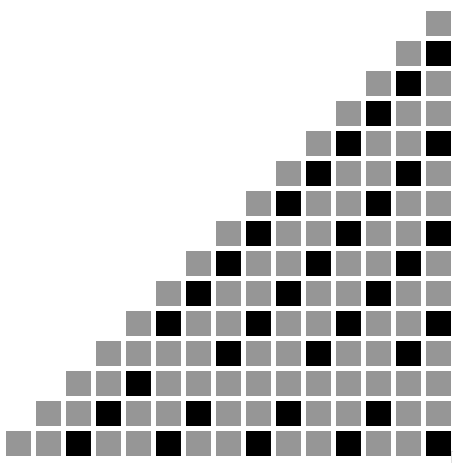}\\
\includegraphics[scale=0.2]{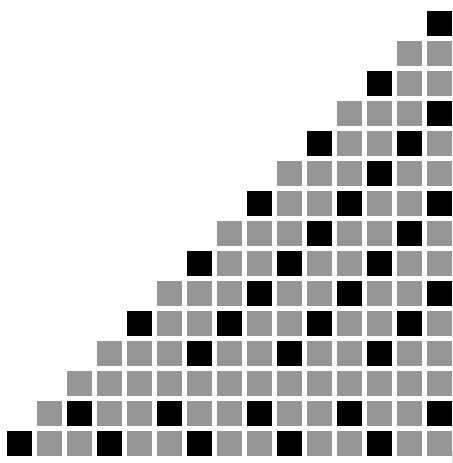} & \includegraphics[scale=0.2]{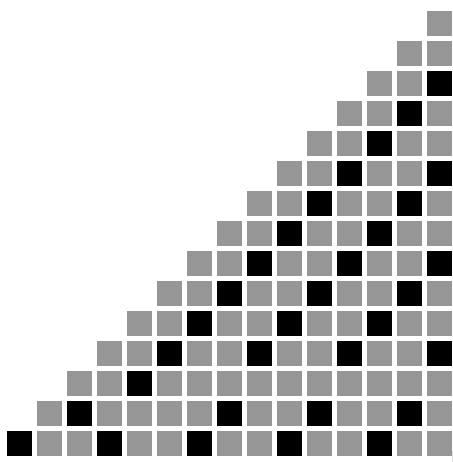} & \includegraphics[scale=0.2]{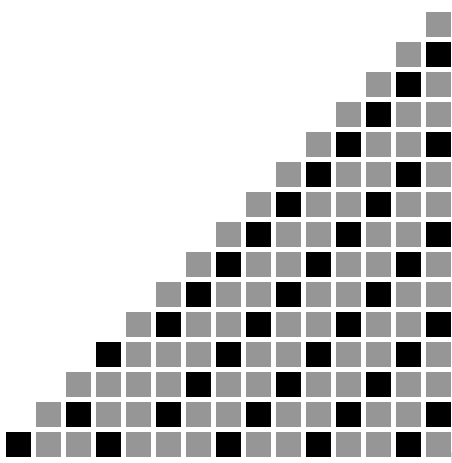} & \includegraphics[scale=0.2]{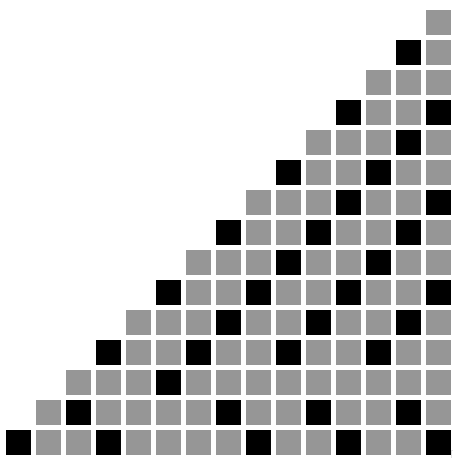}\\
\includegraphics[scale=0.2]{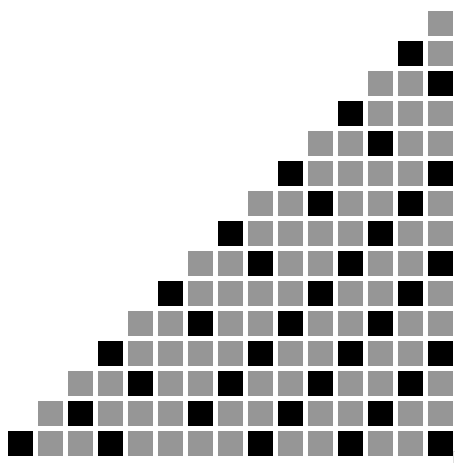} & \includegraphics[scale=0.2]{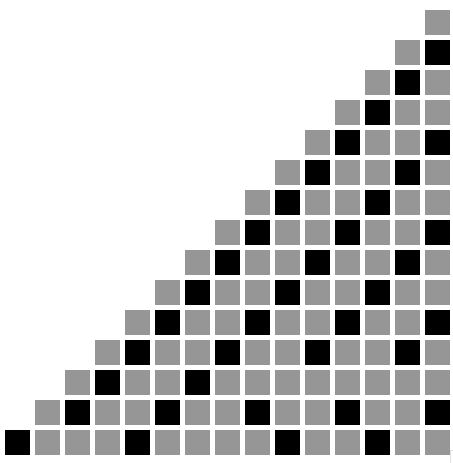} & \includegraphics[scale=0.2]{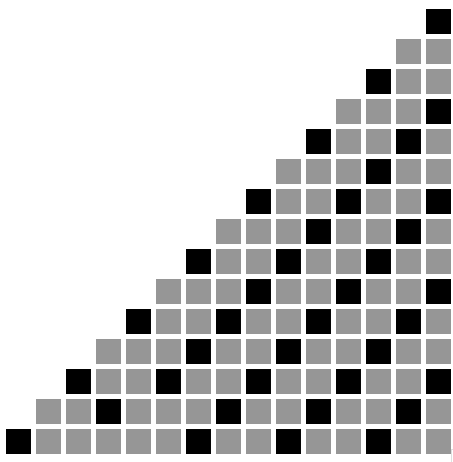} & \includegraphics[scale=0.2]{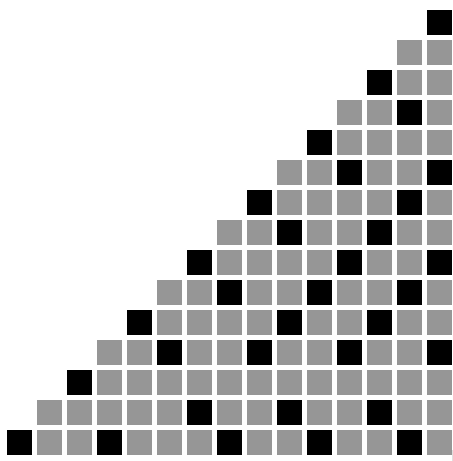}\\
\end{tabular}
\caption{P-positions for 2-Diet Chomp with 3 rows and $a_1$ ranging from 0 to 11}
\label{fig:DMCP-pos}
\end{table}

We can make the following observation from this pictures:

\begin{itemize}
\item Each row eventually becomes periodic with period either 3 or 1.
\item Each diagonal going NE becomes periodic with period either 2 or 1.
\item If we remove the left bottom corner, a few bottom rows and a few top NE diagonals, the pictures would look the same, and the P-positions correspond to values $a_1 +a_3 - a_2 \equiv 1 \pmod 3$.
\end{itemize}

\section{Acknowledgments}

This project was part of the PRIMES STEP program. We are thankful to the program for allowing us the opportunity to conduct this research.


\begin{thebibliography}{9}

\bibitem{AND}
M.~H. Albert, R.~J. Nowakowski, and D.~Wolfe, {\em Lessons in Play}, A.~K.~Peters, Wellesley MA, 2007.

\bibitem{BCG}
Elwyn~R. Berlekamp, John~H. Conway, and Richard~K. Guy, {\em Winning Ways for Your Mathematical Plays}, A.~K.~Peters, Natick MA, 2001.

\bibitem{Bouton}
Charles Bouton, Nim, a game with a complete mathematical theory, {\em The Annals of Mathematics}, 3(14):35--39, 1901.

\bibitem{Z}
Doron Zeilberger, Three-Rowed CHOMP, {\em Adv. Applied Math.} 26 (2001) 168--179. 


\end{thebibliography}
\end{document}